\documentclass[a4paper]{amsart}[11pt]
\def\doctype{}
\usepackage{latexsym,amssymb,amsthm}
\usepackage{color}
\usepackage{dsfont}
\usepackage{fancyhdr}
\usepackage{nicematrix}
\usepackage{tikz}
\usepackage{hyperref}
\usepackage{float}
\usepackage{enumitem}

\newcommand\A{\mathrm{A}}

\renewcommand\S{\mathrm{S}}

\newcommand\fix{\mathrm{fix}}

\newcommand{\comment}[1]{}

\fancypagestyle{titlepage}{
\fancyhead[R]{\doctype}
\fancyhead[CL]{}
\cfoot{\vspace{5pt} \thepage}
}

\let\oldsection\section
\newcommand\boldsection[1]{\oldsection{\bf #1}}
\newcommand\starsection[1]{\oldsection*{\bf #1}}
\makeatletter
\renewcommand\section{\@ifstar\starsection\boldsection}
\makeatother

\newtheoremstyle{algorithm}
  {12pt}		  
  {0pt}  
  {\tt}  
  {\parindent}     
  {\bf}  
  {. }    
  {\newline}    
  {}     
\theoremstyle{algorithm}

\newtheoremstyle{theorem}
  {12pt}		  
  {0pt}  
  {\sl}  
  {\parindent}     
  {\bf}  
  {. }    
  { }    
  {}     
\theoremstyle{theorem}
\newtheorem{thm}{Theorem}[section]  
\newtheorem{lemma}[thm]{Lemma}     

\newtheorem{prop}[thm]{Proposition}

\newtheoremstyle{definition}
  {12pt}		  
  {0pt}  
  {}  
  {\parindent}     
  {\bf}  
  {. }    
  { }    
  {}     
\theoremstyle{definition}

\renewcommand{\proofname}{Proof}

\makeatletter
\renewenvironment{proof}[1][\proofname]{\par
  \pushQED{\qed}%
  \normalfont \partopsep=\z@skip \topsep=\z@skip
  \trivlist
  \item[\hskip\labelsep
        \scshape
    #1\@addpunct{.}]\ignorespaces
}{%
  \popQED\endtrivlist\@endpefalse
}
\makeatother

\makeatletter
\renewcommand*\@maketitle{%
  \normalfont\normalsize
  \@adminfootnotes
  \@mkboth{\@nx\shortauthors}{\@nx\shorttitle}%
  \global\topskip0\p@\relax 
  \@settitle
  \ifx\@empty\authors \else {\vskip 1em
\vtop{\centering\shortauthors\@@par}} \fi
  \ifx\@empty\@date \else {\vskip 1em \vtop{\centering\@date\@@par}}\fi 
  \ifx\@empty\@dedicatory
  \else
    \baselineskip18\p@
    \vtop{\centering{\footnotesize\itshape\@dedicatory\@@par}%
      \global\dimen@i\prevdepth}\prevdepth\dimen@i
  \fi
  \@setabstract
  \normalsize
  \if@titlepage
    \newpage
  \else
    \dimen@34\p@ \advance\dimen@-\baselineskip
    \vskip\dimen@\relax
  \fi
} 
\renewcommand*\@adminfootnotes{%
  \let\@makefnmark\relax  \let\@thefnmark\relax
  \ifx\@empty\@subjclass\else \@footnotetext{\@setsubjclass}\fi
  \ifx\@empty\@keywords\else \@footnotetext{\@setkeywords}\fi
  \ifx\@empty\thankses\else \@footnotetext{%
    \def\par{\let\par\@par}\@setthanks}%
  \fi
\thispagestyle{titlepage}
}
\makeatother

\pgfsetlayers{nodelayer,edgelayer} 
\usepackage{graphicx}

\begin{document}
\title[]{\large A character theoretic\\ formula for base size}

\author{Coen del Valle}
\address{
School of Mathematics and Statistics,
University of St Andrews, St Andrews, UK
}
\email{cdv1@st-andrews.ac.uk}

\thanks{The author is grateful to Pablo Spiga for bringing this problem to his attention. The author would also like to thank his supervisors Colva Roney-Dougal and Peter Cameron for their support and guidance as well as Peiran Wu for interesting early discussions. Research of Coen del Valle is supported by the Natural Sciences and Engineering Research Council of Canada (NSERC), [funding reference number PGSD-577816-2023], as well as a University of St Andrews School of Mathematics and Statistics Scholarship.}

\date{\today}

\begin{abstract}
A base for a permutation group $G$ acting on a set $\Omega$ is a sequence $\mathcal{B}$ of points of $\Omega$ such that the pointwise stabiliser $G_{\mathcal{B}}$ is trivial. The base size of $G$ is the size of a smallest base for $G$. We derive a character theoretic formula for the base size of a class of groups admitting a certain kind of irreducible character. Moreover, we prove a formula for enumerating the non-equivalent bases for $G$ of size $l\in\mathbb{N}$. As a consequence of our results, we present a very short, entirely algebraic proof of the formula of Mecenero and Spiga~\cite{MeSp} for the base size of the symmetric group $\S_n$ acting on the $k$-element subsets of $\{1,2,3,\dots,n\}$. Our methods also provide a formula for the base size of many product-type permutation groups. 
\end{abstract}

\maketitle
\hrule

\bigskip

\section{Introduction}
A \emph{base} for a permutation group $G$ acting on a finite set $\Omega$ is a sequence $\mathcal{B}$ of points of $\Omega$ with trivial pointwise stabiliser in $G$. The size $b(G)$ of a smallest base for $G$ is called the \emph{base size} of $G$. In 1992, Blaha ~\cite{blaha} showed that the problem of finding a minimum base for an arbitrary group $G$ is NP-hard. Despite this, much work has been done towards determining the base size of certain families of groups, especially primitive groups.

Let $n$ and $k$ be positive integers with $n>2k$ and let $\S_{n,k}$ be the symmetric group $\S_n$ acting on the $k$-element subsets of $[n]:=\{1,2,\dots,n\}$. Recently, in two independent papers~\cite{dvrd,MeSp} precise formulae were determined for the base size $b(\S_{n,k})$, both of which take remarkably different forms. The two results are proved using predominantly combinatorial techniques, the arguments mostly using the language of graphs and hypergraphs. During the review process, an anonymous referee made an astute observation: the beautiful formula of Mecenero and Spiga~\cite[Theorem 1.1]{MeSp} has an entirely character theoretic interpretation. In particular, after some straightforward algebraic manipulation one derives that if $\mathrm{sgn}$ is the sign character of $\S_n$ and $\chi$ is the permutation character of $\S_{n,k}$, then $$b(\S_{n,k})=\min\{l\in \mathbb{N} : \langle \mathrm{sgn},\chi^l\rangle\ne0\}.$$

Recently, Pablo Spiga asked whether this connection was purely coincidental. This paper gives an entirely character theoretic formula for the base size of groups admitting a certain kind of homomorphism which we call \emph{base-controlling}.

Let $G$ be a permutation group acting faithfully on a finite set $\Omega$. We say that a homomorphism {${\phi:G\to\{1,-1\}}$} is \emph{base-controlling} if for every tuple $\mathcal{A}$ of points of $\Omega$, $\mathcal{A}$ is a base if and only if $\phi(G_{\mathcal{A}})=1$. Note that any base-controlling homomorphism of $G$ is an irreducible character of $G$. Define $\langle - , - \rangle$ to be the standard scalar product of (complex-valued) class functions. That is, for class functions $\varphi_1,\varphi_2$ of $G$, $$\langle\varphi_1,\varphi_2\rangle=|G|^{-1}\sum_{g\in G}\varphi_1(g)\overline{\varphi_2(g)}.$$

Our first main result allows us to enumerate the non-equivalent bases of any given size.
\begin{thm}\label{enum}
Let $G$ be a permutation group acting on a set $\Omega$ with permutation character $\chi$. If $G$ admits a base-controlling homomorphism $\phi$, then $G$ has exactly $\langle\phi,\chi^l\rangle$ regular orbits on $\Omega^l$.
\end{thm}
As an immediate consequence to Theorem~\ref{enum}, we deduce a character theoretic formula for base size, generalising~\cite[Theorem 3.1]{MeSp}.
\begin{thm}\label{main}
Let $G$ be a permutation group with permutation character $\chi$. If $G$ admits a base-controlling homomorphism $\phi$, then $b(G)=\min\{l \in\mathbb{N}: \langle\phi,\chi^l\rangle\ne 0\}.$
\end{thm}

Theorem~\ref{enum} has another powerful consequence. Let $G$ be as in the set-up of Theorem~\ref{enum}, and let $P\leq\S_m$ for some $m$. Define $D(P)$ to be the \emph{distinguishing number} of $P$, that is, the minimum number of parts in a partition of $[m]$ such that the pointwise stabiliser in $P$ of these parts is trivial. In ~\cite{bc} it is shown that $b(G\wr P)$ (with product action) is the minimum $l$ such that the number of regular orbits of $G$ on $\Omega^l$ is at least $D(P)$, hence we deduce the following corollary to Theorem ~\ref{enum}.
\begin{thm}\label{wreath}
Let $G$ be a permutation group with permutation character $\chi$, and let $P\leq \S_m$ for some integer $m$. If $G$ admits a base-controlling homomorphism $\phi$, then $$b(G\wr P)=\min\{l\in \mathbb{N} : \langle \phi, \chi^l\rangle\geq D(P)\}.$$
\end{thm}
In Section~\ref{2} we prove Theorem~\ref{enum}, hence deducing Theorems~\ref{main} and~\ref{wreath}. In Section~\ref{3} we discuss some examples and non-examples of groups admitting base-controlling homomorphisms, in particular, we recover the formula for $b(\S_{n,k})$ of Mecenero and Spiga~\cite{MeSp}, and deduce sharp bounds on the base size of large base groups.

\section{Proving Theorem~\ref{enum}}\label{2}
Throughout this section let $G$ be a permutation group acting faithfully on a finite set $\Omega$ with permutation character $\chi$. Suppose that $G$ admits a base-controlling homomorphism $\phi$, and define $K:=\ker\phi$. Given $l\in\mathbb{N}$, let $o(l)$ and $o_K(l)$ be the numbers of $G$-orbits and $K$-orbits on $\Omega^l$, respectively. 

\begin{lemma}\label{interpretation}
Fix $l\in\mathbb{N}$. Then $\langle\phi,\chi^l\rangle=o_K(l)-o(l)$.
\end{lemma}
\begin{proof}
By the Orbit-Counting Lemma, $$o(l)=|G|^{-1}\sum_{g\in G}\fix_{\Omega^l}(g)=|G|^{-1}\left(\sum_{g\in K}\chi(g)^l\right)+|G|^{-1}\left(\sum_{g\in G\setminus K}\chi(g)^l\right),$$ whilst $$o_K(l)=|K|^{-1}\left(\sum_{g\in K}\chi(g)^l\right)=2\cdot|G|^{-1}\left(\sum_{g\in K}\chi(g)^l\right).$$ Therefore, \begin{equation*}\label{orbdif}o_K(l)-o(l)=\left(|G|^{-1}\sum_{g\in K}\chi(g)^l\right)-\left(|G|^{-1}\sum_{g\in G\setminus K}\chi(g)^l\right)=|G|^{-1}\sum_{g\in G}\phi(g)\overline{\chi(g)^l}=\langle\phi,\chi^l\rangle,\end{equation*} as desired.
\end{proof}

By Lemma~\ref{interpretation}, to prove Theorem~\ref{enum} it only remains to show that $o_K(l)-o(l)$ is the number of regular $G$-orbits on $\Omega^l$.

\begin{lemma}\label{lem2}
Fix $l\in\mathbb{N}$. Then $o_K(l)-o(l)$ is the number of regular $G$-orbits on $\Omega^l$.
\end{lemma}
\begin{proof}
Since $|G:K|=2$ each orbit of $G$ on $\Omega^l$ is the union of either one or two $K$-orbits. Thus, $o_K(l)-o(l)$ is the number of $G$-orbits which are the unions of exactly two $K$-orbits. Since every $K$-orbit has fewer than $|G|$ points, it suffices to show that any $G$-orbit which is the union of two $K$-orbits is regular.

Suppose that $\mathcal{A}^G$ is the union of two $K$-orbits. Then $G_{\mathcal{A}}=K_{\mathcal{A}}\leq K$, but $\phi$ is base-controlling, whence $\mathcal{A}$ is a base for $G$. That is, $\mathcal{A}^G$ is a regular $G$-orbit, as was to be shown.
\end{proof}

Putting together Lemmas~\ref{interpretation} and~\ref{lem2} we deduce Theorem~\ref{enum}.

\section{Examples and applications}\label{3}
In this section we present some examples and non-examples of groups admitting base-controlling homomorphisms, and we give an application of Theorem~\ref{wreath}.

We begin by noting that the character $\mathrm{sgn}$ of $\S_{n,k}$ is base-controlling. Indeed, the pointwise stabiliser of any collection of $k$-subsets is a direct product of symmetric groups --- the only symmetric group containing no odd permutations is the trivial group. Thus, we deduce from Theorem~\ref{main} the formula of Mecenero and Spiga.
\begin{thm}{\cite[Theorem 1.1]{MeSp}}
Let $n>2k$ be positive integers. Then $b(S_{n,k})$ is the minimum $l$ such that \begin{align*}\label{eq:5}
\sum_{\substack{\pi\vdash n\\
\pi=(1^{c_1},2^{c_2},\ldots,n^{c_n})}}(-1)^{n-\sum_{i=1}^nc_i}\frac{n!}{\prod_{i=1}^ni^{c_i}c_i!}\left(\sum_{\substack{\eta\vdash k\\
\eta=(1^{b_1},2^{b_2},\ldots,k^{b_k})}}\prod_{j=1}^k{c_j\choose b_j}\right)^l&\ne 0.
\end{align*}
\end{thm}
\begin{proof}
Since $\mathrm{sgn}$ is base-controlling, $b(\S_{n,k})$ is the minimum $l$ such that $\langle \mathrm{sgn},\chi^l\rangle\ne 0$. Let $C_1,C_2,\dots, C_m$ be the conjugacy classes of $\S_{n,k}$, and let $g_i$ be a representative of the class $C_i$ for each $i$. Then $$\langle \mathrm{sgn},\chi^l\rangle=n!^{-1}\sum_{i=1}^m\mathrm{sgn}(g_i)|C_i|\chi(g_i)^l.$$ Now, recall that the conjugacy classes of $\S_{n,k}$ are in correspondence with the partitions of $n$ --- it is easy to see that the sign of a permutation $g$ corresponding to the partition $\pi=(1^{c_1},2^{c_2},\ldots,n^{c_n})$ is exactly $(-1)^{n-\sum_{i=1}^nc_i}$. Moreover, it is well-known that the size of the conjugacy class of such an element is precisely $\frac{n!}{\prod_{i=1}^ni^{c_i}c_i!}$. Similarly, one can calculate that $$\chi(g)=\sum_{\substack{\eta\vdash k\\
\eta=(1^{b_1},2^{b_2},\ldots,k^{b_k})}}\prod_{j=1}^k{c_j\choose b_j},$$ hence the result.
\end{proof}

We now point out that admitting a base-controlling homomorphism is a property of a permutation group, not of an abstract group. Indeed, as observed above, $\mathrm{sgn}$ is a base-controlling homomorphism for $\S_{n,k}$, however, it is not the case that $\mathrm{sgn}$ is always base-controlling for $\S_n$. Take for example $\S_{15}$ acting on the partitions of $[15]$ into exactly $3$ parts of size $5$, with permutation character $\chi$. By ~\cite[Theorem 1.1]{ms} this group has base size $3$, however one can compute using {\sf GAP}~\cite{gap} that $\langle \mathrm{sgn},\chi^2\rangle\ne 0$.

Our final application is towards determining the base size of large base groups. Given integers $n$ and $k$, define $\A_{n,k}$ to be the alternating group acting on the $k$-subsets of $[n]$. A permutation group $G$ is \emph{large base} if there exist integers $m$, $r$, and $k$ such that $\A_{m,k}^r \unlhd G \leq \S_{m,k} \wr \S_r$, where the wreath product acts with product action.
\begin{prop}
Let $m$, $r$, and $k$ be integers, and let $\A_{m,k}^r \unlhd G \leq \S_{m,k} \wr \S_r$. Let $\chi_{m-1}$ and $\chi_m$ be the permutation characters of $\S_{m-1,k}$ and $\S_{m,k}$, respectively. Then $$\min\{l \in \mathbb{N} : \langle\mathrm{sgn},\chi_{m-1}^l\rangle_{\S_{m-1}}>0\}\leq b(G)\leq\min\{l \in \mathbb{N} : \langle\mathrm{sgn},\chi_m^l\rangle_{\S_m}\geq r\}.$$ 
\end{prop}
\begin{proof}
We show that $b(\A_{m,k}^r)=\min\{l \in \mathbb{N} : \langle\mathrm{sgn},\chi_{m-1}^l\rangle_{\S_{m-1}}>0\}$, and that $b(\S_{m,k} \wr \S_r)=\min\{l \in \mathbb{N} : \langle\mathrm{sgn},\chi_m^l\rangle_{\S_{m}}\geq r\}$, from which the result will follow. 

Consider first the lower bound. Note that $b(\A_{m,k}^r)=b(\A_{m,k})=b(\S_{m-1,k})$ by \cite[Theorem 1.1]{dvrd}, hence the lower bound holds by Theorem~\ref{main}. That $b(\S_{m,k} \wr \S_r)=\min\{l \in \mathbb{N} : \langle\mathrm{sgn},\chi_m^l\rangle_{\S_{m}}\geq r\}$ follows immediately from Theorem~\ref{wreath} and the fact that $D(\S_r)=r$.
\end{proof}
\begin{rk}
Of course, if more is known about the structure of $G$ then one may apply Theorem~\ref{wreath} directly to obtain a more precise formula for $b(G)$.
\end{rk}

So far the only example of a group which admits a base-controlling homomorphism that we have seen is the symmetric group --- we conclude this section with an alternative example.

Let $G=\mathrm{PSL}_2(7):2$ acting sharply 3-transitively with point stabiliser $7:6$, and hence 2-point stabiliser $6$. Consider now the homomorphism $\phi: G\to\{1,-1\}$ with kernel $\mathrm{PSL}_2(7)$. The map $\phi$ is base-controlling: if $\mathcal{A}$ is such that $\phi(G_\mathcal{A})=1$, then since $\phi(7:6)=\phi(6)=\{1,-1\}$ it follows that $\mathcal{A}$ contains at least three distinct points. Since $G$ is sharply 3-transitive, any triple of distinct points forms a base for $G$, that is, $G_{\mathcal{A}}=1$ as desired. Of course, the above argument hints at a construction of many groups with a similar structure.

\end{document}